\documentclass[11pt,a4paper]{article}
\pdfoutput=1

\usepackage[utf8]{inputenc}
\usepackage[T1]{fontenc}
\usepackage[english]{babel}
\usepackage{microtype}
\usepackage{amsmath}
\usepackage{amsfonts}
\usepackage{amssymb}
\usepackage{amsthm}
\usepackage{amscd}
\usepackage{latexsym}
\usepackage{thmtools}
\usepackage{mathtools} 
\usepackage{lmodern}
\usepackage{stmaryrd}
\usepackage{MnSymbol}

\usepackage{hyperref}
\hypersetup{
    colorlinks=true,
    linkcolor=blue,
    urlcolor=RoyalBlue,
    citecolor=ForestGreen
    }
    
\usepackage[noabbrev,capitalize,nameinlink]{cleveref}
\crefname{prop}{Proposition}{Propositions}
\crefname{fact}{Fact}{Facts}

\usepackage[titles]{tocloft}
\usepackage[shortlabels]{enumitem}
\setlist[enumerate,1]{label={\upshape\alph*)}}
\setlist[enumerate,2]{label={\upshape\roman*)}}

\usepackage{xfrac}
\usepackage{graphicx}
\usepackage[dvipsnames]{xcolor} 
\usepackage{slashed}
\usepackage[autostyle=true]{csquotes}
\usepackage[all,cmtip]{xy}
\usepackage{tikz-cd}
\usepackage{cancel} 
\usepackage{xparse}

\usepackage[a4paper,
            top=2.5cm,
            bottom=2.5cm,
            inner=3.5cm,
            outer=3.5cm]{geometry}

\usepackage[citestyle=alphabetic,
            bibstyle=alphabetic,
            backend=bibtex,
            url=true,
            doi=false,
            isbn=false,
            giveninits=true]{biblatex}

\bibliography{Bibliography_Higson_exact}
\AtEveryBibitem{%
  \ifentrytype{article}{
    \clearfield{url}%
    \clearfield{urldate}%
  }{}
  \ifentrytype{book}{
    \clearfield{url}%
    \clearfield{urldate}%
  }{}
  \ifentrytype{incollection}{
    \clearfield{url}%
    \clearfield{urldate}%
  }{}
} 

\declaretheorem[name=Theorem,numberwithin=section]{thm}
\declaretheorem[name=Lemma,numberlike=thm]{lem}
\declaretheorem[name=Corollary,numberlike=thm]{cor}
\declaretheorem[name=Proposition,numberlike=thm]{prop}

\declaretheorem[name=Definition,numberlike=thm,style=definition,qed=\(\lozenge\)]{defn}
\declaretheorem[name=Example,numberlike=thm,style=definition,qed=\(\lozenge\)]{ex}
\declaretheorem[name=Remark,numberlike=thm,style=definition,qed=\(\lozenge\)]{rem}

\declaretheorem[name=Fact,numberlike=thm,style=definition,qed=\(\lozenge\)]{fact}

\newcommand{\IN}{\mathbb{N}}

\newcommand{\IZ}{\mathbb{Z}}

\newcommand{\IR}{\mathbb{R}}

\newcommand{\IC}{\mathbb{C}}

\newcommand{\toprm}{\mathrm{top}}

\newcommand{\Mat}{\mathrm{Mat}}
\newcommand{\prob}{\mathrm{prob}}
\newcommand{\op}{\mathrm{op}}

\newcommand{\EM}{\mathrm{EM}}
\newcommand{\BC}{\mathrm{BC}}

\newcommand{\cF}{\mathcal{F}}

\newcommand{\cK}{\mathcal{K}}
\newcommand{\cB}{\mathcal{B}}
\newcommand{\cQ}{\mathcal{Q}}
\newcommand{\cM}{\mathcal{M}}

\newcommand{\EG}{\underline{EG}}

\newcommand*{\Lin}{\cB}

\newcommand{\frakc}{\mathfrak{c}}
\newcommand{\frakcbar}{\bar{\mathfrak{c}}}

\newcommand*{\sHigCom}{\frakcbar}
\newcommand*{\sHigCor}{\frakc}
\newcommand*{\sHigComRed}{\bar{\mathfrak{c}}^{\mathrm{red}}}
\newcommand*{\sHigCorRed}{\mathfrak{c}^{\mathrm{red}}}

\newcommand*{\red}{\mathrm{red}}

\NewDocumentCommand{\textCstar}{}{\ensuremath{\mathrm{C}^*\!}}

\newcommand{\id}{\mathrm{id}}

\NewDocumentCommand{\blank}{}{{-}}

\numberwithin{equation}{section} 

\author{Alexander Engel\thanks{Institut f{\"u}r Mathematik und Informatik der Universit{\"a}t Greifswald, Walther-Rathenau-Str.\,47, 17489 Greifswald, Germany, 
\href{mailto:alexander.engel@uni-greifswald.de}{alexander.engel@uni-greifswald.de}}
}

\title{Groups acting amenably on their Higson corona}

\date{}

\begin{document}

\maketitle

\begin{abstract}
We investigate groups that act amenably on their Higson corona (also known as bi-exact groups) and we provide reformulations of this in relation to the stable Higson corona, nuclearity of crossed products and to positive type kernels.

We further investigate implications of this in relation to the Baum--Connes conjecture, and prove that Gromov hyperbolic groups have isomorphic equivariant $K$-theories of their Gromov boundary and their stable Higson corona.
\end{abstract}

\setcounter{tocdepth}{2}
\tableofcontents

\paragraph{Acknowledgements.}
I had fruitful discussions with the following people: Alcides Buss, Siegfried Echterhoff, Narutaka Ozawa, Rufus Willett and Christopher Wulff.

I acknowledge financial support by the Deutsche Forschungsgemeinschaft DFG through the Priority Programme SPP 2026 ``Geometry at Infinity'' (EN 1163/5-1, project number 441426261, Macroscopic invariants of manifolds).

\section{Introduction}

This paper arose from pursuing a side idea of \cite{anastructinj}, namely to investigate groups that act amenably on their own (reduced) stable Higson corona $\sHigCorRed G$. The motivation comes from the commutative diagram
\begin{equation}\label{eq_intro}
\xymatrix{
K_*^\toprm(G;\sHigCorRed G) \ar[rr]^-{\mu^*_\EM} \ar[dr]_-{\mu_*^\BC} && K_G^{1-*}(\EG)\\
& K_*(\sHigCorRed G \rtimes_\mu G) \ar[ur]_-{\mu^*_G} & 
}
\end{equation}
where $\mu^*_\EM$ is the co-assembly map of Emerson and Meyer, $\mu_*^\BC$ is the Baum--Connes assembly map, $\mu^*_G$ is the equivariant coarse co-assembly map, and where we have by definition $K^{1-*}_G(\EG) \coloneqq K_{*-1}(C_0(\EG) \rtimes_\red G)$ \cite[Sec.~5.1]{anastructinj}.

In the diagram we may use any crossed product functor $-\rtimes_\mu G$ which is exact, and one of our questions was: In which cases does the choice not matter, i.e. when do $\sHigCorRed G \rtimes_{\max} G$ and $\sHigCorRed G \rtimes_\red G$ coincide (assuming for simplicity that $G$ is exact)? As is now known, this question is intimately connected to amenability of the $G$-$C^*$-algebra $\sHigCorRed G$ \cite{BEW_amenable}. We tried to answer this question in \cite[Props.~5.8 \& 5.12]{anastructinj}, but unfortunately, there is an error in the proofs and one of the goals of this paper is to provide an erratum to it.\footnote{Fortunately, since this question was just a side hustle in \cite{anastructinj}, none of the main results of it are affected by this.} We will discuss the error in detail in \cref{sec_reduced_case}.

As it turns out, the corrected version of \cite[Prop.~5.8]{anastructinj} uses the \emph{unreduced} stable Higson compactification $\sHigCom G$, resp.~corona. Concretely, we prove the following (we only state the version for the compactification):
\begin{thm}[Part of \cref{prop_equiv_amenable_stable_compactification}]
\label{thm_equiv_amenable_stable_compactification}
Let $G$ be a countable and discrete group. Then the following are equivalent to each other:
\begin{enumerate}
\item The group $G$ acts amenably on its Higson compactification $hG$.
\item $\sHigCom G$ is an amenable $G$-\textCstar-algebra.
\item We have $\sHigCom G \rtimes_{\max} G \cong \sHigCom G \rtimes_{\red} G$ and $G$ is exact.
\end{enumerate}
\end{thm}

Groups acting amenably on their Higson compactification $hG$ have been already studied before, but with another focus: One can prove that this condition is equivalent to the group $G$ being bi-exact (see \cref{def_biexact}) and from this we get a plethora of examples. This will be quickly summarized in \cref{sec_biexact}.

In \cref{sec_nuclear} we will find more reformulations for a group to act amenably on its Higson compactification. Concretely, we prove the following:
\begin{thm}[\cref{prop_equiv_amenable_nuclear_kernels}]
\label{thm_equiv_amenable_nuclear_kernels}
Let $G$ be a countable and discrete group. Then the following are equivalent to each other:
\begin{enumerate}
\item The group $G$ acts amenably on its Higson compactification.
\item The $C^*$-algebra $C_h(G) \rtimes_\red G$ is nuclear.
\item The embedding $\IC \rtimes_\red G \to C_h(G) \rtimes_\red G$ is nuclear.
\item There is a sequence $(k_n)_{n \in \IN}$ in $C_c(G \times G, \Delta)$ of normalized positive type kernels having vanishing variation on diagonals and converging to $1$ uniformly on all finite width neighbourhoods of the diagonal $\Delta$ in $G \times G$.
\end{enumerate}
\end{thm}

As we will discuss in \cref{rem_between_amenable_exact}, the above result shows that the condition of acting amenably on its Higson compactification sits naturally between amenability and exactness of the group. Because both amenability of exactness have profound implications for the Baum--Connes conjecture for the group (amenability implies bijectivity whereas exactness implies injectivity of the assembly map), the question arises whether one can prove something in this direction for groups acting amenably on their Higson compactifications. We will pursue this line of thought in \cref{sec_isomorphism_results}, and our main result in this direction is the following:
\begin{thm}[\cref{prop_biexact_SESsplit}]
\label{thm_biexact_SESsplit}
Let $G$ be a bi-exact group and assume that it admits a $G$-finite classifying space for proper $G$-actions $\EG$. Then we have the split short exact sequence
\begin{equation}
0 \to K_{*+1}(\sHigComRed \EG \rtimes_\red G) \to K_*(C^*_\red(G)) \to K_*(\sHigCom(\EG) \rtimes_\red G) \to 0.
\end{equation}
Further, the Baum--Connes conjecture for trivial coefficients $\IC$ and coefficients $\sHigComRed \EG$ are equivalent to each other for $G$ and imply the isomorphism
\begin{equation}
K_*(C^*_\red(G)) \xrightarrow{\cong} K_*(\sHigCom(\EG) \rtimes_\red G),
\end{equation}
which is induced from the inclusion of $\cK$ as the constant functions in $\sHigCom(\EG)$.
\end{thm}

Our final result is a computation of $K_*(\sHigCor(\EG) \rtimes_\red G)$ for Gromov hyperbolic groups (note that this result now refers to the stable Higson corona, contrary to the compactification as above in \cref{thm_biexact_SESsplit}):
\begin{thm}[\cref{ex_gromov_iso}]
Let $G$ be a finitely generated, Gromov hyperbolic group. Then we have an isomorphism
\begin{equation}
K_*(C(\partial G) \rtimes_\red G) \cong K_*(\sHigCor G \rtimes_\red G),
\end{equation}
where $\partial G$ denotes the Gromov boundary of $G$.
\end{thm}

\section{Bi-exact groups}
\label{sec_biexact}

Let $G$ be a countable and discrete group. If needed, we will equip it without further mentioning with any proper, left-invariant metric.

\begin{defn}[Amenable actions, {\cite[Def.\ 2.1 \& 2.2]{higson_roe_exact} and \cite[Sec.\ 2.2]{anantharaman_renault}}]
\label{defn_amenable_action}
\mbox{}
\begin{enumerate}
\item For any countable set $Z$ we denote by $\prob(Z)$ the set of Borel probability measures on $Z$, i.e., the set of functions $b\colon Z \to [0,1]$ such that $\sum_{z \in Z} b(z) = 1$.

We view $\prob(Z)$ as a subset of $\ell^1(Z)$ and equip it with the weak-${}^*$-topology (recall that $\ell^1(Z)$ is the Banach space dual of $c_0(Z)$). By $\|\blank\|_1$ we denote the usual norm on $\ell^1(Z)$.

If $Z$ is equipped with an action of the group $G$, then the induced action of $G$ on $\prob(Z)$ is defined by $g.b(z) \coloneqq b(g^{-1}.z)$.
\item Let $X$ be a compact Hausdorff space on which $G$ acts by homeomorphisms. The action is called amenable if there exists a sequence of weak-${}^*$-continuous maps $b^n\colon X \to \prob(G)$ such that for every $g \in G$ we have
\[
\lim_{n \to \infty} \sup_{x \in X} \| g.b^n_x - b^n_{g.x}\|_1 = 0\,.\qedhere
\]
\end{enumerate}
\end{defn}

\begin{defn}[Bi-exact groups {\cite[Ch.~15]{BrownOzawaCstarFiniteDim}}]
\label{def_biexact}
We consider the $G \times G$-action on the group $G$ given by left and right translations.

We call $G$ \emph{bi-exact} if the induced action of $G \times G$ on the Stone-\v{C}ech boundary $\partial_\beta G$ is amenable.
\end{defn}

\begin{rem}
It is known that $G$ is bi-exact if and only if it acts amenably on its Higson corona (which will be defined in \cref{sec_vanishing_var}) \cite[Prop.~15.2.3]{BrownOzawaCstarFiniteDim}. Because the latter is the class of groups that we consider in this paper, everything we prove here holds for bi-exact groups.
\end{rem}

\begin{ex}
The following groups are bi-exact:
\begin{enumerate}
\item Amenable groups
\item Groups hyperbolic relative to a family of amenable subgroups \cite[Prop.~12]{Ozawa2006}
\item Discrete subgroups of connected, simple Lie groups of rank one \cite[Sec.~4]{skandalis_biexact}
\item $\IZ^2 \rtimes \mathrm{SL}(2,\IZ)$ \cite{ozawa_example}
\end{enumerate}
The class of bi-exact groups also enjoys the following closure properties:
\begin{enumerate}
\item It is closed under passage to subgroups.
\item It is closed under free products (with finite amalgamation) \cite[Thm.~3.3]{ozawa_amenactions}
\item Wreath products $\Upsilon \wr G$, where $\Upsilon$ is amenable and $G$ bi-exact, are again bi-exact \cite[Cor.~15.3.9]{BrownOzawaCstarFiniteDim}.
\item It is closed under measure equivalence \cite{sako_ME}.\qedhere
\end{enumerate}
\end{ex}

\begin{rem}
Let us note that the class of bi-exact groups is in general not closed under products: The group $\IZ \times \mathrm{SL}(2,\IZ)$ is not bi-exact.

It follows that $\mathrm{SL}(3,\IZ)$ is not bi-exact either: The group $\IZ \times \mathrm{SL}(2,\IZ)$ ME-embeds into it\footnote{The ME-embedding comes from a discrete
embedding of $\IZ \times \mathrm{SL}(2,\IZ)$ into $\mathrm{SL}(3,\IR)$.} and hence by Sako's result\footnote{In the version that if $G$ ME-embeds into a bi-exact group, then $G$ is bi-exact.} \cite{sako_ME} bi-exactness of $\mathrm{SL}(3,\IZ)$ would propagate to $\IZ \times \mathrm{SL}(2,\IZ)$.

Weakly amenable groups are in general not bi-exact, because weak amenability is, contrary to bi-exactness, closed under direct products. For example, a product of free groups is weakly amenable but not bi-exact.
\end{rem}

\begin{rem}
The importance of bi-exact groups stems from the fact that the group von Neumann algebra of a bi-exact group is solid \cite{ozawa_solid}.
\end{rem}

\section{Amenable actions on the Higson compactification}

\subsection{Functions of vanishing variation}
\label{sec_vanishing_var}

\begin{defn}[Higson compactification and corona]
Let $Y$ be any metric space. If $\vartheta\colon Y \to M$ is a map to another metric space $M$, then for each $r > 0$ the $r$-variation of $\vartheta$ is defined as the function
\[
\mathrm{Var}_r \vartheta\colon Y \to [0,\infty)\,,\quad y \mapsto \sup\{d(\vartheta(y),\vartheta(x)) \colon x \in Y \text{ with } d(y,x) \le r\}\,.
\]
The function $\vartheta$ is said to have vanishing variation if for all $r>0$ the $r$-variation $\mathrm{Var}_r \vartheta$ converges to zero at infinity.\footnote{If this is defined to mean that for every $\varepsilon > 0$ exists a compact subset $K \subset Y$ with $|\mathrm{Var}_r \vartheta(y)| < \varepsilon$ for all $y \in Y \setminus K$, then one should assume $Y$ to be locally compact. If $Y$ is not locally compact, then one should instead demand $K$ to be just bounded.}

The \emph{Higson compactification $hY$} of $Y$ is the Gelfand dual of the $C^*$-algebra of all complex-valued, bounded, continuous functions of vanishing variation on $Y$. The \emph{Higson corona $\partial_h Y$} is defined as the complement $hY \setminus Y$.
\end{defn}

Because the Higson compactification $hG$ is a compact Hausdorff space, acting amenably on it implies that $G$ is exact in the usual sense (\cite{OZAWA2000691,higson_roe_exact} and also \cite[Thm.\ 5.3]{BEW_amenable}).

In the case of the Stone--\v{C}ech compactification and corona the following lemma is well-known, see e.g.\ \cite[Prop.\ 4.4]{BCL_weaklyamen_exact}.
\begin{lem}\label{lem_equiv_amenable_action_compactification_corona}
The group $G$ acts amenably on its Higson compactification $hG$ if and only if it acts amenably on its Higson corona $\partial_h G$.
\end{lem}
\begin{proof}
It is clear from the definition of amenable actions that if $G$ acts amenably on a compact space $X$ and if $Y \to X$ is any equivariant and continuous map of compact $G$-spaces, then the action of $G$ on $Y$ is also amenable \cite[Exer.\ 4.4.3]{BrownOzawaCstarFiniteDim}. Applying this to the inclusion map $\partial_h G \to hG$, we conclude that if $G$ acts amenably on its Higson compactification then it also acts amenably on its Higson corona.

To prove the reverse implication, we assume that $G$ acts amenably on its Higson corona. Since this is a compact Hausdorff space, $G$ is exact. Now we use \cite{matsumura} and \cite[Thm.\ 5.2]{BEW_amenable}: If $G$ is a discrete and exact group and $A$ is a commutative $G$-$C^*$-algebra then the action of $G$ on $A$ being amenable\footnote{Meaning that the action of $G$ on the Gelfand dual of $A$ is amenable \cite[Rem.\ 2.2]{BEW_amenable}.} is equivalent to the canonical quotient map $A \rtimes_{\max} G \to A \rtimes_{\red} G$ being an isomorphism.
Hence to prove that $G$ acts amenably on its Higson compactification it suffices to show that the equality $C(\partial_h G) \rtimes_{\max} G = C(\partial_h G) \rtimes_{\red} G$ implies $C(hG) \rtimes_{\max} G = C(hG) \rtimes_{\red} G$. To this end we consider the commutative diagram
\begin{align*}
\xymatrix{
0 \ar[r] & C_0(G) \rtimes_{\max} G \ar[r] \ar[d]^-{\cong} & C(hG) \rtimes_{\max} G \ar[r] \ar[d] & C(\partial_h G) \rtimes_{\max}G \ar[r] \ar[d]^-{\cong} & 0\\
0 \ar[r] & C_0(G) \rtimes_{\red} G \ar[r] & C(hG) \rtimes_{\red} G \ar[r] & C(\partial_h G) \rtimes_{\red}G \ar[r] & 0
}
\end{align*}
where
\begin{itemize}
\item the top row is exact since $-\rtimes_{\max}$ is an exact functor,
\item the bottom row is exact since we already know that $G$ is exact,
\item the vertical maps are the canonical quotient maps from the maximal to the reduced crossed product,
\item the left vertical map is always an isomorphism (i.e.\ without any assumptions on the group $G$), see e.g.\ \cite[Rem.\ 3.4.16]{echterhoff_KK_BC_overview}, and
\item the right vertical map is an isomorphism by assumption.
\end{itemize}
The five lemma implies that the middle vertical map is also an isomorphism.
\end{proof}

In the previous proof we used the equivalence of acting amenably on a compact Hausdorff space and the weak containment property\footnote{That is to say, the resp.\ maximal and reduced crossed products are isomorphic to each other.} for the Gelfand dual together with exactness of the group. For completeness, we state this now explicitly in the following lemma:
\begin{lem}\label{lem_amenably_containment}
Let $G$ be a countable, discrete group.
\begin{enumerate}
\item\label{item_amenably_compactification} The following are equivalent:
\begin{itemize}
\item $G$ acts amenably on its Higson compactification.
\item $C(hG) \rtimes_{\max} G = C(hG) \rtimes_{\red} G$ and $G$ is exact.
\end{itemize}
\item\label{item_amenably_corona} The following are equivalent:
\begin{itemize}
\item $G$ acts amenably on its Higson corona.
\item $C(\partial_h G) \rtimes_{\max} G = C(\partial_h G) \rtimes_{\red} G$ and $G$ is exact.
\end{itemize}
\end{enumerate}
(Note that by \cref{lem_equiv_amenable_action_compactification_corona} the \cref{item_amenably_compactification,item_amenably_corona} are equivalent to each other.)
\end{lem}

\subsection{Groups acting amenably on the stable Higson compactification, respectively on the stable Higson corona}
\label{sec_amenable_action_stable_comp}

For a discrete group $G$ we will first recall the different notions of amenability of $G$-\textCstar-algebras from \cite[Def.~2.1~\&~4.13]{BEW_amenable} and then apply them to the (unreduced) stable Higson compactification, resp.\ corona (see \cref{defn_stable_Higson_comp_corona} below).
Note that there are also variants of some of these notions occuring in, e.g., \cite{anantharaman,BrownOzawaCstarFiniteDim}. How these variants are related to each other is explained in \cite[Rem.~2.2]{BEW_amenable}.

\begin{defn}\label{defn_amenabilities}
Let $G$ be a discrete group.
\begin{enumerate}
\item\label{item_strongly_amenable} The $G$-\textCstar-algebra $A$ is called strongly amenable if there is a net
\[
(\theta_i\colon G \to Z\mathcal{M}(A))_{i \in I}\,,
\]
where $Z\mathcal{M}(A)$ is the center of the multiplier algebra, of positive type functions\footnote{In general, a function $\vartheta\colon G \to B$ is of positive type if for any finite subset $\{g_1, \ldots, g_n\}$ of $G$ the matrix $(\alpha_{g_i}(\vartheta(g_i^{-1} g_j)))_{i,j} \in M_n(B)$ is positive, where $\alpha$ is the action of $G$ on $B$ \cite[Def.~2.1]{anantharaman_annalen}.} such that
\begin{itemize}
\item each $\theta_i$ is finitely supported,
\item for each $i$ we have $\theta_i(e) \le 1$, and
\item for each $g \in G$ we have $\theta_i(g) \to 1$ strictly as $i \to \infty$.
\end{itemize}
\item\label{item_amenable} The $G$-\textCstar-algebra $A$ is called amenable if there is a net $(\theta_i\colon G \to Z(A^\ast{}^\ast))_{i \in I}$ of positive type functions such that
\begin{itemize}
\item each $\theta_i$ is finitely supported,
\item for each $i$ we have $\theta_i(e) \le 1$, and
\item for each $g \in G$ we have $\theta_i(g) \to 1$ ultra-weakly as $i \to \infty$.\footnote{Recall that a net $(T_\lambda)_{\lambda \in \Lambda}$ in $A^\ast{}^\ast$ converges ultra-weakly to $T$ if and only if $(T_\lambda(\varphi))_{\lambda \in \Lambda}$ converges to $T(\varphi)$ for every $\varphi \in A^\ast$.}
\end{itemize}
\item\label{item_commutant_amenable} The $G$-\textCstar-algebra $A$ is called commutant amenable if for every covariant pair $(\pi,u)\colon (A,G) \to \Lin(H)$ there exists a net $(\theta_i\colon G \to \pi(A)')_{i \in I}$ of positive type functions such that 
\begin{itemize}
\item each $\theta_i$ is finitely supported,
\item for each $i$ we have $\theta_i(e) \le 1$, and
\item for each $g \in G$ we have $\theta_i(g) \to 1$ ultra-weakly as $i \to \infty$.
\end{itemize}
\end{enumerate}
Note that strong amenability implies amenability \cite[Rem.~2.2]{BEW_amenable}, and amenability implies commutant amenability \cite[Rem.\ 4.14]{BEW_amenable}.
\end{defn}

The main players of this section are the (unreduced\footnote{There are also reduced versions of these $C^*$-algebras, but the main results of the present section (\cref{prop_equiv_amenable_stable_compactification,prop_equiv_amenable_stable_corona}) do not hold for them; see the next \cref{sec_reduced_case} for a discussion.}) stable Higson compactification $\sHigCom G$, resp.\ corona $\sHigCor G$ of $G$ as defined in the following \cite[Defn.\ 3.2]{EmeMey}:
\begin{defn}\label{defn_stable_Higson_comp_corona}
Let $G$ be a countable and discrete group, and equip it with any proper, left-invariant metric.\footnote{Since any two such metrics are coarsely equivalent to each other, the defined $C^*$-algebras are independent of this choice.} Fix any separable Hilbert space $H$.

The \emph{(unreduced) stable Higson compactification $\sHigCom G$} is the $C^*$-algebra of all bounded (continuous) functions of vanishing variation $f\colon G \to \cK(H)$.

The \emph{(unreduced) stable Higson corona $\sHigCor G$} is defined as the quotient
\[
\sHigCor G \coloneqq \sHigCom G / C_0(G,\cK(H))\,.\qedhere
\]
\end{defn}

\noindent The following answers the variant of \cite[Ques.\ 5.10]{anastructinj} for the unreduced stable Higson compactification in the case of the metric space acted on by $G$ being the group $G$ itself:
\begin{prop}\label{prop_equiv_amenable_stable_compactification}
Let $G$ be a countable and discrete group. Then the following are equivalent to each other:
\begin{enumerate}
\item\label{item_question_amenable_one} The group $G$ acts amenably on its Higson compactification $hG$.
\item\label{item_question_amenable_two} $\sHigCom G$ is an amenable ($hG \rtimes G$)-\textCstar-algebra.\footnote{That is to say, there is a $G$-equivariant, non-degenerate ${}^*$-homomorphism $\phi\colon C(hG) \to Z\mathcal{M}(\sHigCom G)$, and $G$ acts amenably on $hG$ \cite[Defn.\ 6.1]{anantharaman}.}
\item\label{item_question_amenable_three} $\sHigCom G$ is a strongly amenable $G$-\textCstar-algebra.
\item\label{item_question_amenable_four} $\sHigCom G$ is an amenable $G$-\textCstar-algebra.
\item\label{item_question_amenable_five} $\sHigCom G$ is a commutant amenable $G$-\textCstar-algebra.
\item\label{item_question_amenable_six} We have $\sHigCom G \rtimes_{\max} G \cong \sHigCom G \rtimes_{\red} G$ and $G$ is exact.
\end{enumerate}
\end{prop}
\begin{proof}
To see the equivalence \textbf{\ref{item_question_amenable_one}$\boldsymbol\Leftrightarrow$\ref{item_question_amenable_two}} we need to provide the $G$-equivariant, non-degenerate map $\phi\colon C(hG) \to Z\mathcal{M}(\sHigCom G)$. To this end we observe that $\mathcal{M}(\sHigCom G)$ is the $G$-\textCstar-algebra of all bounded (continuous) functions of vanishing variation $G\to \cB(H)$, and hence $Z\mathcal{M}(\sHigCom G) \cong C(hG)$. Under this isomorphism the map $\phi$ is the identity.

The implication \textbf{\ref{item_question_amenable_two}$\boldsymbol\Rightarrow$\ref{item_question_amenable_three}} is a general fact \cite[Lem.\ 2.5]{BEW_amenable}, and the same is true for \textbf{\ref{item_question_amenable_three}$\boldsymbol\Rightarrow$\ref{item_question_amenable_four}} by \cite[Rem.~2.2]{BEW_amenable} and for \textbf{\ref{item_question_amenable_four}$\boldsymbol\Rightarrow$\ref{item_question_amenable_five}} by \cite[Rem.\ 4.14]{BEW_amenable}.

Let us show \textbf{\ref{item_question_amenable_four}$\boldsymbol\Rightarrow$\ref{item_question_amenable_one}}.
Firstly, note that there is a $C^*$-embedding $C(hG) \to \sHigCom G$ by fixing any rank-one projection $p$ on $H$ and mapping $f \in C(hG)$ to $f \otimes p \in \sHigCom G$. It extends to a normal\footnote{This means that for every bounded, increasing net $(x_i)$ of positive elements we have $\phi(\sup x_i) = \sup \phi(x_i)$ \cite[Defn.\ III.2.2.1]{blackadar_operator_algebras}.} $C^*$-embedding $\phi\colon C(hG)^{\ast\ast} \to (\sHigCom G)^{\ast\ast}$ \cite[Sec.\ III.5.2.10]{blackadar_operator_algebras}.
Secondly, there exists a conditional expectation from $\sHigCom G$ onto the image of $C(hG)$ therein arising from the vector state on $H$ corresponding to (the unit vector defined by) the rank-one projection $p$. It extends to a normal conditional expectation
\begin{equation}\label{eq_conditional_expectation_to_hG}
\psi\colon (\sHigCom G)^{\ast\ast} \to \phi(C(hG)^{\ast\ast}) \cong C(hG)^{\ast\ast}\,.
\end{equation}
Let us argue that $\psi$ is unital. As a general fact, the double dual $(\sHigCom G)^{\ast\ast}$ contains the multiplier algebra $\cM(\sHigCom G)$ as a subalgebra: It is the idealizer of $\sHigCom G$ in $(\sHigCom G)^{\ast\ast}$; and hence the unit of $\cM(\sHigCom G)$ is the one of $(\sHigCom G)^{\ast\ast}$. From this description of the unit we see that $\psi$ is unital.
From \ref{item_question_amenable_four} we get a net $(\theta_i\colon G \to Z((\sHigCom G)^\ast{}^\ast))_{i \in I}$ of positive type functions satisfying the properties in \cref{defn_amenabilities}.\ref{item_amenable}.
We compose with $\psi$ to get a net $(\psi\circ \theta_i\colon G \to C(hG)^{\ast\ast})_{i \in I}$\footnote{We use here that conditional expectations map central elements to central elements, and that we have $C(hG)^{\ast\ast} = Z(C(hG)^{\ast\ast})$ since $C(hG)^{\ast\ast}$ is commutative.}
showing amenability of $C(hG)$, where we used that $\psi$ is normal and hence continuous for the ultra-weak topologies \cite[Prop.\ III.2.2.2]{blackadar_operator_algebras}. Because $C(hG)$ is commutative, amenability of $C(hG)$ is equivalent to amenability of the $G$-action on the space $hG$ \cite[Rem.~2.2]{BEW_amenable}.

Let us show the implication \textbf{\ref{item_question_amenable_five}$\boldsymbol\Rightarrow$\ref{item_question_amenable_one}}.
We first note that $\sHigCom G$ is $G$-equivariantly $C^*$-isomorphic to its opposite $(\sHigCom G)^\op$ by applying a $C^*$-isomorphism $\cB(H) \to \cB(H)^\op$ point-wise to functions in $\sHigCom G$.
Next we employ the standard form representations for von Neumann algebras as developed by Haagerup \cite{haagerup_standard_forms} in the form presented in \cite[Thm.\ 5.1]{BEW_amenable}: There exists a normal, unital and faithful representation $\pi^\op\colon ((\sHigCom G)^\op)^{\ast\ast} \to \cB(V)$ on a Hilbert space $V$ and a unitary representation $u$ of $G$ on $V$ such that $(\pi^\op,u)$ is a covariant pair and we have $\pi^\op((\sHigCom G)^\op)^\prime \cong (\sHigCom G)^{\ast\ast}$. Composing $\pi^\op$ with an equivariant $C^*$-isomorphism $\sHigCom G \cong (\sHigCom G)^\op$ we get a covariant pair $(\rho,u)$ for $\sHigCom G$ with the property that $\rho(\sHigCom G)^\prime \cong (\sHigCom G)^{\ast\ast}$.
To this covariant pair we can now apply the assumed \ref{item_question_amenable_five} to get a net $(\theta_i\colon G \to \rho(\sHigCom G)')_{i \in I}$ of positive type functions having the properties listed in \cref{defn_amenabilities}.\ref{item_commutant_amenable}. Using the isomorphism $\rho(\sHigCom G)^\prime \cong (\sHigCom G)^{\ast\ast}$ and composing with the map $\psi$ from \eqref{eq_conditional_expectation_to_hG} we conclude that $C(hG)$ is amenable. This implies that $G$ acts amenably on $hG$ \cite[Rem.~2.2]{BEW_amenable}.

Finally, let us prove the equivalence \textbf{\ref{item_question_amenable_five}$\boldsymbol\Leftrightarrow$\ref{item_question_amenable_six}}. That \ref{item_question_amenable_five} implies the weak containment property $\sHigCom G \rtimes_{\max} G \cong \sHigCom G \rtimes_{\red} G$ is a general fact: It is the implication (i)$\boldsymbol\Rightarrow$(ii) in \cite[Thm.\ 4.17]{BEW_amenable} (this implication does not need the assumption in the cited theorem that $G$ is exact). Using the already proven implication \ref{item_question_amenable_five}$\boldsymbol\Rightarrow$\ref{item_question_amenable_one} and since $hG$ is a compact Hausdorff space, we conclude that \ref{item_question_amenable_five} also implies that $G$ is exact. The reverse implication \ref{item_question_amenable_six}$\boldsymbol\Rightarrow$\ref{item_question_amenable_five} is true in general \cite[Thm.\ 4.17]{BEW_amenable}.
\end{proof}

Let us now quickly turn to the corresponding statements for the (unreduced) stable Higson corona (thus resolving the unreduced variant of \cite[Conj.\ 1.25]{anastructinj} in the case that the metric space is the group itself):
\begin{prop}\label{prop_equiv_amenable_stable_corona}
Let $G$ be a countable and discrete group. Then the following are equivalent to each other:
\begin{enumerate}
\item\label{item_question_amenable_one_corona} The group $G$ acts amenably on its Higson corona $\partial_h G$.
\item\label{item_question_amenable_two_corona} $\sHigCor G$ is an amenable ($\partial_h G \rtimes G$)-\textCstar-algebra.
\item\label{item_question_amenable_three_corona} $\sHigCor G$ is a strongly amenable $G$-\textCstar-algebra.
\item\label{item_question_amenable_four_corona} $\sHigCor G$ is an amenable $G$-\textCstar-algebra.
\item\label{item_question_amenable_five_corona} $\sHigCor G$ is a commutant amenable $G$-\textCstar-algebra.
\item\label{item_question_amenable_six_corona} We have $\sHigCor G \rtimes_{\max} G \cong \sHigCor G \rtimes_{\red} G$ and $G$ is exact.
\end{enumerate}
\end{prop}
\begin{proof}
The proof is analogous to the one of \cref{prop_equiv_amenable_stable_compactification}. One only has to tweak the computation of the multiplier algebra, which now results in $\mathcal{M}(\sHigCor G)$ being the quotient of the $G$-\textCstar-algebra of all the bounded (continuous) functions of vanishing variation $G\to \cB(H)$ by its ideal $C_0(G,\cB(H))$, and one has to check that the map $\psi$ from \eqref{eq_conditional_expectation_to_hG} descends to the corresponding quotients.
\end{proof}

By \cref{lem_equiv_amenable_action_compactification_corona}
the group $G$ acts amenably on its Higson compactification $hG$ if and only if it acts amenably on its Higson corona $\partial_h G$. We therefore conclude the following corollary:
\begin{cor}
The conditions in \cref{prop_equiv_amenable_stable_compactification} are equivalent to the conditions in \cref{prop_equiv_amenable_stable_corona}.
\end{cor}

\subsection{The case of the reduced algebras}
\label{sec_reduced_case}

\begin{defn}[{\cite[Def.~5.4]{EmeMey}}]
\label{defn_red_stable_Higson_comp_corona}
Let $G$ be any countable and discrete group, and equip it with any proper, left-invariant metric. Fix a separable Hilbert space $H$.

The \emph{reduced} stable Higson compactification $\sHigComRed G$ is the $C^*$-algebra of all bounded (continuous) functions of vanishing variation $f\colon G \to \cB(H)$ with $f(x)-f(y) \in \cK(H)$ for all $x,y\in G$.

The \emph{reduced} stable Higson corona $\sHigCorRed G$ is defined as the quotient
\[
\sHigCor G \coloneqq \sHigComRed G / C_0(G,\cK(H))\,.\qedhere
\]
\end{defn}

Contrary to what was claimed in \cite[Prop. 5.8]{anastructinj} the analogue of \cref{prop_equiv_amenable_stable_compactification} for the reduced stable Higson compactification $\sHigComRed G$ is in general not true (and similarly for Proposition 5.12 in loc.\,cit.), as we will discuss below. This also answers in the negative Question 5.10 and Conjecture 1.25 in loc.\,cit.

\begin{lem}\label{lem_counterex}
Assume that $\sHigComRed G$ is an amenable $G$-\textCstar-algebra. Then $G$ is amenable.
\end{lem}
\begin{proof}
Note that $\sHigCom G$ is an ideal in $\sHigComRed G$ with quotient the Calkin algebra $\cB(H) / \cK(H)$ equipped with the trivial $G$-action. Since amenability descends to quotients, the claim follows since the trivial action can only be amenable for amenable groups.
\end{proof}

The analogue of \cref{lem_counterex} for the reduced stable Higson corona is also true by the same argument.

\begin{ex}
Consider a Gromov hyperbolic group $G$. It is known that $G$ acts amenably on its Gromov boundary \cite[Ex.~2.7.4]{anantharaman} and hence, because the Gromov boundary is Higson dominated, $G$ acts also amenably on its Higson corona.

Because hyperbolic groups are in general not amenable, this provides concrete counter-examples to \cite[Prop. 5.8 \& 5.12]{anastructinj} and invalidates Example 5.13 and Proposition 1.24 in loc.\,cit.
\end{ex}

\begin{rem}
The mistake in \cite{anastructinj} occurs in Lemma 5.6 therein: The map constructed there is in general not unital, contrary to what is claimed there. But the unitality of this map was crucial for the proof of Proposition 5.8 therein.
\end{rem}

\subsection{Nuclearity of crossed products and positive type kernels}
\label{sec_nuclear}

Let us first recall the necessary notions related to kernels and Schur multipliers:

\begin{defn}[{\cite[Sec.~11.2]{roe_lectures_coarse_geometry}}]
A symmetric function $k\colon G \times G \to \IR$ is called a \emph{positive type kernel}, if for every $n \in \IN$ and every $g_1, \ldots, g_n \in G$ the matrix given by $[k(g_i,g_j)]_{i,j} \in \Mat_{n \times n}(\IR)$ is positive semidefinite.

\begin{itemize}
\item A positive type kernel $k$ is called \emph{normalized} if $k(g,g) = 1$ for every $g \in G$.

Note that in this case the kernel will be uniformly bounded; concretely, we will have $|k(g,h)| \le 1$ for all $g,h \in G$ \cite[Thm.~D.3]{BrownOzawaCstarFiniteDim}.
\item A positive type kernel $k$ is called \emph{equivariant} if $k(h_1 g, h_2 g) = k(h_1,h_2)$ for every $g,h_1,h_2 \in G$.\qedhere
\end{itemize}
\end{defn}

If $k$ is a positive type kernel on $G \times G$ with $k(g,g) \le 1$ for every $g \in G$, then the Schur multiplier
\begin{equation}\label{eq_defn_Schur_multipliers}
\theta_k\colon \cB(\ell^2(G)) \to \cB(\ell^2(G))\,, \quad [T_{g,h}]_{g,h \in G} \mapsto [k(g,h) T_{g,h}]_{g,h \in G}
\end{equation}
is a completely positive contraction; and if $k$ is additionally normalized, then the corresponding Schur multiplier $\theta_k$ is unital and completely positive (\cite[Lem.~11.17]{roe_lectures_coarse_geometry}, \cite[Thm.~D.3]{BrownOzawaCstarFiniteDim}).

We will write $C_c(G\times G, \Delta)$ for the algebra of all functions $f$ on $G \times G$ for which there is an $R > 0$ such that $f(g,h) = 0$ whenever $d(g,h) > R$; and we call a subset $E$ of $G \times G$ a \emph{finite width neighbourhood of the diagonal $\Delta$} if there is an $R > 0$ such that $d(x,y) < R$ for all $(x,y) \in E$.

The following important properties that the group $G$ might have were originally defined in different terms, but can be equivalently defined by the existence of positive type functions with certain properties:
\begin{fact}\label{fact_amenable_exact_kernels}\mbox{}
\begin{enumerate}
\item\label{fact_amenable_kernel} The group $G$ is \emph{amenable}, if there is a sequence $(k_n)_{n \in \IN}$ in $C_c(G \times G,\Delta)$ of normalized, equivariant positive type kernels converging to $1$ uniformly on all finite width neighbourhoods of the diagonal $\Delta$ in $G \times G$.

References are \cite[Thm.\ 2.6.8]{BrownOzawaCstarFiniteDim}, \cite[Prop.\ 3.5]{anantharaman} and \cite[Subsec.~4.3]{anatharaman_amen_exact_operatoralgebras}.

\item\label{fact_exact_kernel} The group $G$ is \emph{exact}, if there exists a sequence $(k_n)_{n \in \IN}$ in $C_c(G \times G,\Delta)$ of normalized positive type kernels converging to $1$ uniformly on all finite width neighbourhoods of the diagonal $\Delta$ in $G \times G$.

References are \cite[Thm.\ 3]{OZAWA2000691}, \cite[Lem.\ 11.37]{roe_lectures_coarse_geometry}, \cite[Prop.~3.2]{tu} and \cite[Ex.\ 2.7(1)]{anantharaman}.\qedhere
\end{enumerate}
\end{fact}

Let us introduce in the following a condition on kernels which is related to the Higson compactification:
\begin{defn}
Let $k$ be a function on $G \times G$. We will say that $k$ has \emph{vanishing variation on diagonals} if for every $g \in G$ the function $h \mapsto k(h,gh)$ on $G$ has vanishing variation.
\end{defn}

Obviously, if $k$ is equivariant, then it has vanishing variation on diagonals. This means that in the following \cref{prop_equiv_amenable_nuclear_kernels} the \cref{item_question_kernel_four} sits naturally between the two conditions in \cref{fact_amenable_exact_kernels}:

\begin{prop}\label{prop_equiv_amenable_nuclear_kernels}
Let $G$ be a countable and discrete group. Then the following are equivalent to each other:
\begin{enumerate}
\item\label{item_question_kernel_one} The group $G$ acts amenably on its Higson compactification.
\item\label{item_question_kernel_two} The $C^*$-algebra $C_h(G) \rtimes_\red G$ is nuclear.
\item\label{item_question_kernel_three} The embedding $\IC \rtimes_\red G \to C_h(G) \rtimes_\red G$ is nuclear.
\item\label{item_question_kernel_four} There is a sequence $(k_n)_{n \in \IN}$ in $C_c(G \times G, \Delta)$ of normalized positive type kernels having vanishing variation on diagonals and converging to $1$ uniformly on all finite width neighbourhoods of the diagonal $\Delta$ in $G \times G$.
\end{enumerate}
\end{prop}

\begin{proof}
The equivalence of \ref{item_question_kernel_one} and \ref{item_question_kernel_two} follows from \cite[Thm.~5.8]{anantharaman}. Note that since we assume $G$ to be discrete, the Property (W) in the statement of \cite[Thm.~5.8]{anantharaman} is automatically satisfied by \cite[Ex.~4.4]{anantharaman}.\footnote{The implication \ref{item_question_kernel_one} $\Rightarrow$ \ref{item_question_kernel_two} is also shown in \cite[Theorem on last page]{matsumura} using Lemma \ref{lem_amenably_containment}\ref{item_amenably_compactification}.}

The equivalence of \ref{item_question_kernel_one} with \ref{item_question_kernel_four} follows from \cite[Prop.\ 2.5]{anantharaman} in combination with the reformulation of it discussed directly before Prop.\ 3.5 in loc.~cit.

That \ref{item_question_kernel_two} implies \ref{item_question_kernel_three} is clear: Nuclearity of $C_h(G) \rtimes_\red G$ means that its identity map is nuclear, whence the composition $\IC \rtimes_\red G \to C_h(G) \rtimes_\red G \xrightarrow{\id} C_h(G) \rtimes_\red G$ will be also nuclear \cite[Exer.~2.1.4]{BrownOzawaCstarFiniteDim}.

For the proof that \ref{item_question_kernel_three} implies \ref{item_question_kernel_four} we follow the corresponding proof for exactness in \cite[Lem.~2 \& Thm.~3]{OZAWA2000691}.
By the nuclearity assumption,\footnote{See \cite[Exer.~2.1.1]{BrownOzawaCstarFiniteDim} for this version of nuclearity of maps.} for any finite subset $E \subset G$ (with $e \in E$), regarded as a subset $E \subset \IC \rtimes_\red G$, and $\varepsilon > 0$ there is an $n \in \IN$ and unital completely positive maps $\phi\colon \IC \rtimes_\red G \to \Mat_n(\IC)$ and $\psi\colon \Mat_n(\IC) \to C_h(G) \rtimes_\red G$ such that
\[
\|(\psi \circ \phi)(x) - x\| < \varepsilon / 2 \text{ for all } x \in E\,.
\]
The map $\phi$ can be extended to a unital completely positive map $\cB(\ell^2(G)) \to \Mat_n(\IC)$ along the inclusion $\IC \rtimes_\red G = C^*_\red(G) \subset \cB(\ell^2(G))$ \cite[Cor.\ 1.5.16]{BrownOzawaCstarFiniteDim}; let us continue to denote the extended map by $\phi$. We can now do the approximation argument from the proof of \cite[Lem.~2]{OZAWA2000691} with $\phi$ to finally obtain the unital completely positive map $\phi^{\prime\prime}\colon \IC \rtimes_\red G \to \Mat_n(\IC)$ satisfying
\[
\|\phi^{\prime\prime}(x) - \phi(x)\| < \varepsilon / 2 \text{ for all } x \in E
\]
and the following property: Putting $\theta\coloneqq \psi \circ \phi^{\prime\prime}$ we get a map $\theta$ which is
\begin{enumerate}
\item unital completely positive and of finite rank,
\item satisfies $\|\theta(x) - x\| < \varepsilon$ for all $x \in E$, and
\item has the form
\[
\theta(x) = \sum_{k=1}^d \omega_{\delta_{p(k)},\delta_{q(k)}}(x) \otimes y_k
\]
for elements $y_k \in C_h(G) \rtimes_\red G$ and the linear functionals $\omega_{\delta_{p(k)},\delta_{q(k)}}$ on $\cB(\ell^2(G))$ given by $\omega_{\delta_{p(k)},\delta_{q(k)}}(x) = \langle x(\delta_{p(k)}),\delta_{q(k)}\rangle$ for $p(k),q(k) \in G$.
\end{enumerate}
Setting $k\colon G \times G \to \IC$ as $k(s,t) \coloneqq \langle \delta_s, \theta(s t^{-1}) \delta_t\rangle$, where we have secretely used the canonical inclusion $C_h(G) \rtimes_\red G \subset \cB(\ell^2(G))$, we get a normalized positive type kernel which is
\begin{itemize}
\item supported on the finite width neighbourhood defined by
\[
F \coloneqq \{q(k)p(k)^{-1}\colon k=1,\ldots,d\}
\]
of the diagonal $\Delta \subset G \times G$,\footnote{This is the subset of all $(s,t) \in G \times G$ satisfying $st^{-1} \in F$.}
\item has vanishing variation on diagonals, and
\item which is $\varepsilon$-close to $1$ on the finite width neighbourhood defined by $E$.
\end{itemize}
This finishes the proof that \ref{item_question_kernel_three} implies \ref{item_question_kernel_four}.

The proof of the proposition is now complete. As a remark, let us note that the implication from \ref{item_question_kernel_four} to \ref{item_question_kernel_two} can be directly proven by doing the obvious modifications to the proof of \cite[Thm.~3(ii)$\Rightarrow$(iii)]{OZAWA2000691}
\end{proof}

\begin{rem}\label{rem_between_amenable_exact}
\cref{item_question_kernel_one} in \cref{prop_equiv_amenable_nuclear_kernels} sits naturally between amenability and exactness, and the same is true for \cref{item_question_kernel_four} by \cref{fact_amenable_exact_kernels}.

It is known that nuclearity of the reduced group $C^*$-algebra $C^*_\red(G) = \IC \rtimes_\red G$ is equivalent to amenability \cite{lance_amenability}. Moreover, nuclearity of the uniform Roe algebra (which is isomorphic to $\ell^\infty(G) \rtimes_\red G$) is equivalent to the exactness of $G$ \cite{OZAWA2000691}. Hence \cref{item_question_kernel_two} in \cref{prop_equiv_amenable_nuclear_kernels} sits naturally between amenability and exactness.

In \cite[Rem.~2 in Sec.~5]{Guentner_Kaminker} it was suggested to consider $C^*$-algebras $A$ satisfying $C^*_\red(G) \subset A \subset \ell^\infty(G) \rtimes_\red G$ and impose the requirement that the inclusion of $C^*_\red(G)$ into $A$ be a nuclear map in order to get conditions interpolating between amenability and exactness. The \cref{prop_equiv_amenable_nuclear_kernels} and the further results of this article show that $A = C_h(G) \rtimes_\red G$ is a good choice.
\end{rem}

\section{Isomorphism results}
\label{sec_isomorphism_results}

Let $G$ be a countable, discrete group and assume that it admits a $G$-finite classifying space for proper $G$-actions $\EG$. Then we have a short exact sequence
\begin{equation}\label{eq_SES_sHigComCorRed}
0 \to C_0(\EG)\otimes\cK \to \sHigComRed \EG \to \sHigCorRed G \to 0
\end{equation}
of $G$-$C^*$-algebras.
\begin{prop}\label{prop_LHS_BC_sHigComRed}
Let $G$ be a countable discrete group and assume that it admits a $G$-finite classifying space for proper $G$-actions $\EG$.

If $G$ is exact, then the boundary morphism $\partial\colon K_*^\toprm(G;\sHigCorRed G) \to K_{*-1}^\toprm(G;C_0(\EG))$ resulting from \eqref{eq_SES_sHigComCorRed} is an isomorphism and consequently
\begin{equation}\label{eq_Ktop_vanishes_sHigComRed}
K_{*-1}^\toprm(G;\sHigComRed \EG) \cong 0
\end{equation}
for all $* \in \IZ$.
\end{prop}
\begin{proof}
We have the commutative diagram with bijective top horizontal map 
\begin{equation}
\label{eq_assembly_and_coassembly}
\xymatrix{
K_*^\toprm(G;\sHigCorRed G) \ar[rr]^-{\mu^*_\EM,\cong} \ar[dr]_-{\mu_*^\BC} && K_G^{1-*}(\EG)\\
& K_*(\sHigCorRed G \rtimes_\red G) \ar[ur]_-{\mu^*_G} & 
}
\end{equation}
where $\mu^*_\EM$ is the co-assembly map of Emerson and Meyer, $\mu_*^\BC$ is the Baum--Connes assembly map, $\mu^*_G$ is the equivariant coarse co-assembly map, and where we have by definition $K^{1-*}_G(\EG) \coloneqq K_{*-1}(C_0(\EG) \rtimes_\red G)$ \cite[Sec.~5.1]{anastructinj}.

The short exact sequence \eqref{eq_SES_sHigComCorRed} induces the following commutative diagram whose rows are exact and the vertical maps are the respective assembly maps:
\begin{equation*}
\mathclap{
\xymatrix{
\cdots \ar[r] & K_*^\toprm(G;\sHigCorRed G) \ar[r]^-{\partial} \ar[d] \ar@{-->}[dr]^-{\mu^*_\EM,\cong} & K_{*-1}^\toprm(G;C_0(\EG)) \ar[r] \ar[d]^-{\cong} & K_{*-1}^\toprm(G;\sHigComRed \EG) \ar[r] \ar[d] & \cdots \\
\cdots \ar[r] & K_*(\sHigCorRed G \rtimes_\red G) \ar[r]^-{\partial} & K_{*-1}(C_0(\EG) \rtimes_\red G) \ar[r] & K_{*-1}(\sHigComRed \EG \rtimes_\red G) \ar[r] & \cdots
}
}
\end{equation*}
The middle vertical map is an isomorphism, since $C_0(\EG)$ is a proper $G$-$C^*$-algebra, and the diagonal dashed map is the one from Diagram~\eqref{eq_assembly_and_coassembly}. The claim follows.
\end{proof}

Strengthening the assumption on $G$ from exactness to bi-exactness, we arrive at the following:
\begin{prop}\label{prop_biexact_SESsplit}
Let $G$ be a bi-exact group and assume that it admits a $G$-finite classifying space for proper $G$-actions $\EG$. Then we have a split short exact sequence
\begin{equation}\label{eq_biexakt_split_compactifications}
0 \to K_{*+1}(\sHigComRed \EG \rtimes_\red G) \to K_*(C^*_\red(G)) \to K_*(\sHigCom(\EG) \rtimes_\red G) \to 0.
\end{equation}
Further, the Baum--Connes conjecture for trivial coefficients $\IC$ and coefficients $\sHigComRed \EG$ are equivalent to each other for $G$ and imply the isomorphism
\begin{equation}\label{eq_biexact_BC}
K_*(C^*_\red(G)) \xrightarrow{\cong} K_*(\sHigCom(\EG) \rtimes_\red G),
\end{equation}
which is induced from the inclusion of $\cK$ as the constant functions in $\sHigCom(\EG)$.
\end{prop}
\begin{proof}
We have a short exact sequence of $G$-$C^*$-algebras
\begin{equation}\label{eq_SES_Compactification_Q}
0 \to \sHigCom(\EG) \to \sHigComRed \EG \to \cQ \to 0,
\end{equation}
where $\cQ$ is the Calkin algebra of a separable, $\infty$-dimensional Hilbert space, and $\cQ$ is equipped with the trivial $G$-action. We consider the resulting commutative diagram with exact rows and where the vertical maps are the respective assembly maps:
\begin{equation*}
\mathclap{
\xymatrix{
\cdots \ar[r]^-{\partial} & K_*^\toprm(G;\sHigCom(\EG)) \ar[r] \ar[d] & K_{*}^\toprm(G;\sHigComRed \EG) \ar[r] \ar[d] & K_{*}^\toprm(G;\cQ) \ar[r]^-{\partial} \ar[d] & \cdots \\
\cdots \ar[r]^-{\partial} & K_*(\sHigCom(\EG) \rtimes_\red G) \ar[r] & K_{*}(\sHigComRed \EG \rtimes_\red G) \ar[r] & K_{*}(\cQ \rtimes_\red G) \ar[r]^-{\partial} & \cdots
}
}
\end{equation*}
By \cref{prop_LHS_BC_sHigComRed} we have $K_{*}^\toprm(G;\sHigComRed \EG) \cong 0$ and therefore the boundary maps in the top row are isomorphisms.
Since we assume that $G$ is bi-exact, $\sHigCom(\EG)$ is an amenable $(h\EG \rtimes G)$-$C^*$-algebra and we conclude that the assembly map for it is an isomorphism \cite[Cor.~0.4]{BC_for_amenable}.
We therefore arrive at the diagram
\begin{equation}\label{eq_diagram_biexact_BC}
\xymatrix{
\cdots \ar[r]^-{\partial,\cong} & K_*^\toprm(G;\sHigCom(\EG)) \ar[r] \ar[d]^\cong & 0 \ar[r] \ar[d] & K_{*}^\toprm(G;\cQ) \ar[r]^-{\partial,\cong} \ar[d] & \cdots \\
\cdots \ar[r]^-{\partial} & K_*(\sHigCom(\EG) \rtimes_\red G) \ar[r] & K_{*}(\sHigComRed \EG \rtimes_\red G) \ar[r] & K_{*}(\cQ \rtimes_\red G) \ar[r]^-{\partial} & \cdots
}
\end{equation}
showing that $K_*(\sHigCom(\EG) \rtimes_\red G) \to K_{*}(\sHigComRed \EG \rtimes_\red G)$ is the zero map. The isomorphisms in this diagram provide the split for the resulting short exact sequence
\[
0 \to K_{*+1}(\sHigComRed \EG \rtimes_\red G) \to K_{*+1}(\cQ \rtimes_\red G) \xrightarrow{\partial} K_*(\sHigCom(\EG) \rtimes_\red G) \to 0\,.
\]
The proof of \eqref{eq_biexakt_split_compactifications} is finished with the isomorphism $K_{*+1}(\cQ \rtimes_\red G) \cong K_*(C^*_\red(G))$ which is the boundary map in the long exact sequence induced from the short exact sequence of $G$-$C^*$-algebras $0 \to \cK \to \cB \to \cQ \to 0$ with the trivial $G$-action. To see that \eqref{eq_biexact_BC} is induced from the inclusion of $\cK$ as the constant functions in $\sHigCom(\EG)$ we refer to the proof of Proposition 5.5 in \cite{EmeMey}.

The final statement of the proposition about the Baum--Connes conjecture follows from Diagram \eqref{eq_diagram_biexact_BC}.
\end{proof}

We can now prove an equivariant version of \cite[Prop.~4.5]{WilletHomological}. To state it we need the canonical inclusion $i\colon C(\partial G) \otimes \cK \to \sHigCor G$ from \cite[Prop.~3.6]{EmeMey} for boundaries at infinity $\partial G$ of suitable compactifications of $G$. In the following such a boundary will arise as the boundary of a compactification $\bar P$ of a suitable model $P$ for $\EG$.

\begin{prop}\label{prop_iso_willett}
Let $G$ be a countable, discrete group.

We assume that $G$ admits a $G$-finite model $P$ for its classifying space for proper $G$-actions $\EG$ such that $P$ admits a metrizable compactification $\bar P$ satisfying:
\begin{enumerate}
\item $\bar P$ is Higson-dominated,
\item $\bar P$ is $H$-equivariantly contractible for every finite subgroup $H < G$,
\item the $G$-action on $P$ extends to an amenable action on $\bar P$.
\end{enumerate}

Then the inclusion $i\colon C(\partial G) \otimes \cK \to \sHigCor G$, where $\partial G$ is the boundary of $P$ inside $\bar P$, induces an isomorphism
\begin{equation}\label{eq_iso_boundary_Higson_corona}
K_*(C(\partial G) \rtimes_\red G) \cong K_*(\sHigCor G \rtimes_\red G).
\end{equation}
\end{prop}
\begin{proof}
There is a $G$-equivariant quasi-isometry $G \to P$ (canonical up to closeness) inducing $G$-equivariant $C^*$-isomorphisms $\sHigCor G \to \sHigCor P$ and $\sHigCorRed G \to \sHigCorRed P$. We conclude
\[K_*(\sHigCor G \rtimes_\red G) \cong K_*(\sHigCor P \rtimes_\red G) \quad \text{and} \quad K_*(\sHigCorRed G \rtimes_\red G) \cong K_*(\sHigCorRed P \rtimes_\red G).
\]

In the following we will use that since $G$ acts amenably on a Higson-dominated compactification, it also acts amenably on its Higson compactification and hence $G$ is bi-exact. This implies further that $G$ is exact.

We consider the commutative diagram
\begin{equation*}
\mathclap{
\xymatrix{
0 \ar[r] & (C_0(P) \otimes \cK) \rtimes_{\red} G \ar[r] \ar@{=}[d] & \sHigCom P \rtimes_\red G \ar[r] & \sHigCor P \rtimes_\red G \ar[r] & 0\\
0 \ar[r] & (C_0(P) \otimes \cK) \rtimes_{\red} G \ar[r] & (C(\bar P)\otimes \cK) \rtimes_{\red} G \ar[r] \ar[u] & (C(\partial G)\otimes \cK) \rtimes_{\red} G \ar[r] \ar[u]_-{i^\prime \rtimes_{\red} G} & 0
}
}
\end{equation*}
whose rows are exact since $G$ is an exact group. The map $i^\prime$ is the map $i$ composed with the $G$-equivariant $C^*$-isomorphism $\sHigCor G \to \sHigCor P$. From this diagram and the induced transformation of the corresponding long exact sequences in $K$-theory, together with the previously noted isomorphism $K_*(\sHigCor G \rtimes_\red G) \cong K_*(\sHigCor P \rtimes_\red G)$, we see that to prove \eqref{eq_iso_boundary_Higson_corona} it suffices to prove that the middle vertical map in the above diagram induces isomorphisms in $K$-theory.

Consider the commutative diagram whose horizontal maps are the respective assembly maps and the vertical maps are induced by $\bar{i}^\prime\colon C(\bar P) \otimes \cK \to \sHigCom P$:
\[
\xymatrix{
K_*^\toprm(G;\sHigCom P) \ar[r]^-{\cong} & K_*(\sHigCom P \rtimes_\red G)\\
K_*^\toprm(G;C(\bar P)\otimes \cK) \ar[r]^-{\cong} \ar[u] & K_*((C(\bar P)\otimes \cK) \rtimes_\red G) \ar[u]
}
\]
As already exploited in the proof of \cref{prop_biexact_SESsplit}, since $G$ is bi-exact we know that $\sHigCom P$ is an amenable $(hP \rtimes G)$-$C^*$-algebra and we can conclude that the assembly map for it, which is the top horizontal map in the above diagram, is an isomorphism \cite[Cor.~0.4]{BC_for_amenable}. By the same argument, because we assume that $G$ acts amenably on $\bar P$, we conclude that the lower horizontal map is an isomorphism \cite{tu_UCT}.

We have a short exact sequence of $G$-$C^*$-algebras
\begin{equation}\label{eq_SES_Compactification_Q_hyperbolic}
0 \to \sHigCom P \to \sHigComRed P \to \cQ \to 0\,,
\end{equation}
where $\cQ$ is the Calkin algebra of a separable, $\infty$-dimensional Hilbert space and $\cQ$ is equipped with the trivial $G$-action. We consider the resulting long exact sequence:
\begin{equation*}
\cdots \xrightarrow{\partial} K_*^\toprm(G;\sHigCom P) \to K_{*}^\toprm(G;\sHigComRed P) \to K_{*}^\toprm(G;\cQ) \xrightarrow{\partial} \cdots
\end{equation*}
\cref{prop_LHS_BC_sHigComRed} implies $K_*^\toprm(G;\sHigComRed P) \cong 0$ and we conclude that the boundary map is an isomorphism
\begin{equation}\label{eq_boundary_map_iso}
K_{*+1}^\toprm(G;\cQ) \xrightarrow{\partial,\cong} K_*^\toprm(G;\sHigCom P).
\end{equation}
On the other hand, since $\bar P$ is $H$-equivariantly contractible for every finite subgroup $H$ of $G$ we conclude by \cite[Prop.~3.7]{higson_bivariant_k_theory} that we have the isomorphism
\begin{equation}\label{eq_iso_gromov_compactification}
K_*^\toprm(G;\cK) \xrightarrow{\cong} K_*^\toprm(G;C(\bar P)\otimes \cK)
\end{equation}
induced from the inclusion of $\IC$ into $C(\bar P)$ as constant functions. We check (same as \cite[Proof of Prop.~5.5]{EmeMey}) that the diagram with horizontal maps \eqref{eq_boundary_map_iso} and \eqref{eq_iso_gromov_compactification}
\[
\xymatrix{
K_{*+1}^\toprm(G;\cQ) \ar[r]^-{\partial,\cong} \ar[d]_-{\partial,\cong} & K_*^\toprm(G;\sHigCom P)\\
K_*^\toprm(G;\cK) \ar[r]^-{\cong} & K_*^\toprm(G;C(\bar P)\otimes \cK) \ar[u]
}
\]
commutes, where the left vertical map is the boundary map induced from the short exact sequence $0 \to \cK \to \cB \to \cQ \to 0$ and the right vertical map is the one where we have to show that it is an isomorphism. But this follows from the diagram.
\end{proof}

\begin{ex}\label{ex_gromov_iso}
Let $G$ be a Gromov hyperbolic group. We collect in the following the references which varify the assumptions of \cref{prop_iso_willett} for these groups.

Let $P_d(G)$ be the Rips complex at scale $d \ge 1$ of the group $G$ (equipped with any word metric) and $P$ be its second barycentric subdivision. It is known that for large $d$ this is a model for the classifying space $\EG$ for proper actions of $G$ \cite{meintrup_schick}. There is a $G$-equivariant quasi-isometry $G \to P$ (canonical up to closeness) and hence $P$ is also hyperbolic and its Gromov boundary is canonically equivariantly homeomorphic to the boundary $\partial G$ of $G$.

It is known that the Gromov compactification $\bar P$ is $H$-equivariantly contractible for every finite subgroup $H < G$,\footnote{This is a general result and follows from the Whitehead Theorem for Families \cite[Thm.~1.6]{lueck_survey}: Apply it to $Y = \bar P$, $Z = \mathrm{pt}$, $H$ as the group, and all its subgroups as the family $\cF$. In Point (i) of its statement we choose $X = \bar P$ and get that the set $[X,Y]^H$ of $H$-homotopy classes of $H$-maps consists of a single element. But this set contains both the identity and the map to a single $H$-fixed point.} and it is furthermore known that $G$ acts amenably on $\partial G$ \cite[Ex.~2.7.4]{anantharaman} and therefore also on $\bar P$. By \cite[Cor.~2.2]{roe_hyperbolic} we know that the Gromov compactification is Higson dominated.

Finally, let us mention that hyperbolic groups satisfy the Baum--Connes conjecture for all coefficients \cite{lafforgue_BC} and hence the isomorphism \eqref{eq_biexact_BC} holds true.
\end{ex}

\printbibliography[heading=bibintoc]

\end{document}